\documentclass[11pt]{article}
\usepackage[colorlinks]{hyperref}
\usepackage{color}
\usepackage{graphicx}
\usepackage{graphics}
\usepackage{makeidx}
\usepackage{showidx}
\usepackage{latexsym}
\usepackage{amssymb}
\usepackage{verbatim}
\usepackage{amsmath}
\usepackage{amsthm}
\usepackage{amsfonts}

\newtheorem{Theorem}{Theorem}[section]

\newtheorem{cor}{Corollary}[section]
\newtheorem{lemma}{Lemma}

\title{Contact Structure on the Indicatrix Bundle\\
of Finslerian Warped Product Manifolds}
\author{H. Attarchi and M. M. Rezaii}
\begin{document}
\maketitle
\noindent
\begin{abstract}
In this paper, a comprehensive study of contact and Sasakian structures on the indicatrix bundle of Finslerian warped product manifolds is reconstructed. In addition, the Kahler structure on the tangent bundle of these manifolds is studied for some different metrics. Throughout the paper, the contact structure of indicatrix bundle in warped product Finsler manifolds is presented. It is shown that indicatrix bundle cannot be a Sasakian manifold.
\vspace{.3cm}

\end{abstract}

\section{Introduction}
One of the basic concepts in the present work is the warped product. Bishop and O'Neill~\cite{onil} studied manifolds with negative curvatures and introduced the notion of warped product as a generalization of Riemannian products. Afterwards, warped product was used to model the standard space time, especially in the neighborhood of stars and black holes~\cite{onil2}. In Riemannian geometry, warped product manifolds and submanifolds were studied in numerous works~\cite{bejancu1,chen}. In~\cite{Rezaii,kosma}, the concept of Finslerian warped product was developed. In this work, some properties of Finsler manifolds are expanded to the warped product Finsler manifolds. First, with respect to importance of contact structure on the indicatrix bundle which is studied from different aspects~\cite{balan}, the contact structure of indicatrix bundle of warped product Finsler manifolds is reconstructed as a development of the work~\cite{bejancu2008} on Finsler manifolds. Then, the conditions of being a Kahler manifold on tangent bundle of Finslerian warped product manifolds are presented as a generalization of the works~\cite{opiro1,payghan}. In addition, it is also found that indicatrix bundle cannot be a Sasakian manifold in $TM$ for Finsler manifolds. These concepts are very closed and in some viewpoints even related to each other. The contact geometry in many ways is an odd-dimensional counterpart of symplectic geometry, which belongs to the even-dimensional world. Both contact and symplectic geometry are motivated by the mathematical formalism of classical mechanics in which either the even-dimensional phase space of a mechanical system or the odd-dimensional extended phase space including the time variable is considered. Also, Sasakian structures and Kahlerian structures are known as dual structures belong to odd  dimensional and even-dimensional geometry, respectively~\cite{Blair}.

Aiming at finding the relations of these structures on a warped product Finsler manifold, this paper is organized in the following way. In section 2, warped product on the Finsler manifolds is reviewed with respect to the idea of Finslerian warped product manifolds in~\cite{Rezaii,kosma}. In Section 3, the triple $(\varphi,\eta,\xi)$ is introduced which is compatible with the Sasakian metric on $TM$ in the sense of contact metric definition~\cite{Blair}. Then, it is proved that this structure is the contact structure of indicatrix bundle for warped product Finsler manifolds. This result is a generalization of proposition 4.1 in~\cite{bejancu2008} for warped product Finsler manifolds. In Section 4, the Kahler structure of a warped product Finsler manifold is studied with respect to the sasaki and Oproiu's metrics~\cite{opiro1}. Section 4 of this work is similar to the study conducted on Finsler manifolds in~\cite{payghan}. After proving that that indicatrix bundle naturally has contact structure, it will cross the minds that\\
"\emph{Does it happen that $I\!M$ under some conditions has Sasakian structure?}"\\
Finally in section 5, The components of Levi-Civita connection on the indicatrix bundle of a Finsler manifold for the restricted Sasakian metric are calculated. To this end, a new frame is introduced in which $TTM$ can be written in the form of direct sum of vertical Liouville vector field and tangent bundle of indicatrix bundle. In the rest of this section, it is proved that the indicatrix bundle cannot have a Sasakian structure with the lifted sasaki metric $G$ and natural almost complex structure $J$ on tangent bundle. Therefore, we should not have any expectation of properties of Sasakian manifolds in Riemannian geometry on indicatrix bundle as a Riemannian submanifold of $(TM,G,J)$.

\section{Preliminaries and Notations}
Suppose that $(M_1,F_1)$ and $(M_2,F_2)$ are two Finsler manifolds with $\dim M_1=n$ and $\dim M_2=m$, respectively. Let $M=M_1\times M_2$ be the product of these two Finslerian manifolds and consider the function $F:TM^0\longrightarrow\mathbb{R}^+$ given by
\begin{equation}~\label{F}
F=\sqrt{F_1^2+f^2F_2^2}
\end{equation}
where $TM^0:=TM_1^0\oplus TM_2^0$ and $f:M_1\longrightarrow\mathbb{R}$, is Fundamental function of Finslerian manifold $M$. It was proved in~\cite{Rezaii} that $(M,F)$ is a Finsler manifold called warped product Finsler manifold. The indices $\{i,j,k,\ldots\}$ and $\{\alpha,\beta,\lambda,\ldots\}$ are used here for the ranges $1,\ldots,n$ and $1,\ldots,m$, respectively. In addition, the indices $\{a,b,c,\ldots\}$ are used for the range $1,\ldots,n+m$.
First, a review of some formulas in~\cite{Rezaii} is presented. The natural basis of Finslerian manifolds $M_1$ and $M_2$ is given as follows:
\begin{equation}~\label{basis1}
\left\{
\begin{array}{l}
TTM_1=<\frac{\delta}{\delta x^i},\frac{\partial}{\partial y^i}>, \ \ \ \ T^*TM_1=<dx^i,\delta y^i> \cr
TTM_2=<\frac{\delta}{\delta u^{\alpha}},\frac{\partial}{\partial v^{\alpha}}>, \ \ \ \ T^*TM_2=<du^{\alpha},\delta v^{\alpha}>
\end{array}
\right.
\end{equation}
It is supposed that $g_{ij}=\frac{1}{2}\frac{\partial^2F_1^2}{\partial y^i\partial y^j}$ and $g_{\alpha\beta}=\frac{1}{2}\frac{\partial^2F_2^2}{\partial v^{\alpha}\partial v^{\beta}}$. Then, the Hessian matrix of $F$ with respect to the coordinate $(\textbf{x},\textbf{y})\in TM$ is shown by $(g_{ab})$ and expressed as follows
\begin{equation}~\label{metricF}
(g_{ab})=\left(\frac{1}{2}\frac{\partial^2F^2}{\partial\mathbf{y}^a\partial\mathbf{y}^b}\right)=
\left(%
\begin{array}{cc}
  \frac{1}{2}\frac{\partial^2F_1^2}{\partial y^i\partial y^j}& 0 \\
  0 & \frac{f^2}{2}\frac{\partial^2F_2^2}{\partial v^{\alpha}\partial v^{\beta}}\\
  \end{array}%
\right)=
\left(%
\begin{array}{cc}
  g_{ij} & 0 \\
  0 & f^2g_{\alpha\beta}\\
  \end{array}%
\right)
\end{equation}
where $\mathbf{x}=(x,u)\in M_1\times M_2$, $\mathbf{y}=(y,v)\in T_xM_1\oplus T_uM_2$ and $\mathbf{y}^a=\delta_i^ay^i+\delta_{\alpha+n}^av^{\alpha}$.

If the semi-spray coefficients of $F_1$, $F_2$ and $F$ are shown by $G^i$, $G^{\alpha}$ and $B^a$, respectively, then the following equations can be obtained
\begin{equation}~\label{B i}
B^i=G^i-\frac{1}{4}g^{ij}F_2^2\frac{\partial f^2}{\partial x^j}\ ,\ \ \ B^{\alpha+n}=G^{\alpha}+\frac{1}{4f^2}g^{\alpha\beta}\frac{\partial F_2^2}{\partial v^{\beta}}\frac{\partial f^2}{\partial x^i}y^i
\end{equation}
For simplicity, $B^{\alpha+n}$ is denoted by $B^{\alpha}$. In addition, nonlinear connection coefficient of $F$ is shown as follows

\centerline{$
B_b^a:=
\left(%
\begin{array}{cc}
  B_j^i & B_j^{\alpha} \\
  B_{\beta}^i & B_{\beta}^{\alpha}\\
  \end{array}%
\right)
$}
\noindent where $B_j^i=\frac{\partial B^i}{\partial y^j},\ B_j^{\alpha}=\frac{\partial B^{\alpha}}{\partial y^j}, B_{\beta}^i=\frac{\partial B^i}{\partial v^{\beta}}, B_{\beta}^{\alpha}=\frac{\partial B^{\alpha}}{\partial v^{\beta}}$ and thus:
\begin{equation}~\label{Bba}
\left\{
\begin{array}{l}
B_j^i=G_j^i-\frac{1}{4}F_2^2\frac{\partial f^2}{\partial x^h}\frac{\partial g^{ih}}{\partial y^j}\\
B_{\beta}^i=-\frac{1}{4}g^{ij}\frac{\partial f^2}{\partial x^j}\frac{\partial F_2^2}{\partial v^{\beta}}\\
B_j^{\alpha}=\frac{1}{4f^2}g^{\alpha\beta}\frac{\partial f^2}{\partial x^j}\frac{\partial F_2^2}{\partial v^{\beta}}\\
B_{\beta}^{\alpha}=G_{\beta}^{\alpha}+\frac{1}{2f^2}\frac{\partial f^2}{\partial x^j}y^j\delta_{\beta}^{\alpha}
\end{array}
\right.
\end{equation}
Here, some contents and notations which they are needed in the rest of this work are presented. First, indicatrix bundle of Finsler manifold $(M,F)$ is denoted by $I\!M$ and defined in~\cite{bao-chern-shen} by
$$I\!M=\{(x,y)\in TM\ | \ F(x,y)=1\}.$$
A contact structure on a $(2n+1)$-dimensional Riemannian manifold $(M,g)$ is a triple $(\varphi,\eta,\xi)$, where $\varphi$ is a $(1,1)$-tensor, $\eta$ a global 1-form and $\xi$ a vector field, such that:

\centerline{$\varphi(\xi)=\eta\circ\varphi=0, \ \ \ \ \ \ \eta(\xi)=1,\ \ \ \ \ \ \varphi^2=-Id+\eta\otimes\xi$}
\centerline{$g(\varphi X,\varphi Y)=g(X,Y)-\eta(X)\eta(Y), \ \ \ \ \ \ \ d\eta(X,Y)=g(X,\varphi(Y))$}
\noindent for all $X,Y\in\Gamma(TM)$~\cite{Blair}.

\section{Contact Structure of Indicatrix Bundle on the Warped Product Finsler Manifolds}
Let $(M,F)$ be a warped product Finsler manifold. In this section, a local frame of tangent and cotangent bundle on $TM$ is choosen which is suitable for studying the contact structure on $I\!M$ as follows:
\begin{eqnarray}~\label{M basis1}
\left\{
\begin{array}{l}
TTM=<\frac{\delta^*}{\delta^*x^i},\ \frac{\delta^*}{\delta^*u^{\alpha}},\ \frac{\partial}{\partial y^i},\ \frac{\partial}{\partial v^{\alpha}}>\cr
T^*TM=<dx^i,\ du^{\alpha},\ \delta^*y^i,\ \delta^*v^{\alpha}>
\end{array}
\right.
\end{eqnarray}
where
$$\left\{
\begin{array}{l}
\frac{\delta^*}{\delta^*x^i}:=\frac{\partial}{\partial x^i}-B_i^j\frac{\partial}{\partial y^j}-B_i^{\alpha}\frac{\partial}{\partial v^{\alpha}},\ \delta^*y^i:=dy^i+B_j^idx^j+B_{\alpha}^idu^{\alpha}
\cr
\frac{\delta^*}{\delta^*u^{\alpha}}:=\frac{\partial}{\partial u^{\alpha}}-B_{\alpha}^{\beta}\frac{\partial}{\partial v^{\beta}}-B_{\alpha}^i\frac{\partial}{\partial y^i},\ \delta^*v^{\alpha}:=dv^{\alpha}+B_{\beta}^{\alpha}du^{\beta}+B_i^{\alpha}dx^i
\end{array}
\right.$$
which are well-defined by considering Proposition 1 in~\cite{Rezaii}. In addition, with respect to these bases the Sasakian lift of the fundamental tensor $g_{ab}$ in~(\ref{metricF}) on $TM$ is shown by $G$ and given by:
\begin{equation}~\label{metric}
G:=g_{ij}dx^idx^j+f^2g_{\alpha\beta}du^{\alpha}du^{\beta}+g_{ij}\delta^*
y^i\delta^*y^j+f^2g_{\alpha\beta}\delta^*v^{\alpha}\delta^*v^{\beta}
\end{equation}
The \emph{vertical Liouville} vector field $L$ of $M$ which is perpendicular to the level curve of $F$ can be calculated as follows:
\begin{eqnarray}~\label{gradiant}
L:=grad F=\frac{grad F^2}{2F}=\frac{y^i}{F}\frac{\partial}{\partial y^i}+\frac{v^{\alpha}}{F}\frac{\partial}{\partial v^{\alpha}}
\end{eqnarray}
It can be seen that $G(L,L)=1$. The dual 1-form $\eta$ of vertical Liouville vector field $L$ is equal to $dF$ or can be calculated by $\eta(X)=G(L,X)$ which can be locally presented as follows:
\begin{equation}~\label{diff F2}
\eta:=dF=\frac{dF^2}{2F}=\frac{y^j}{F}g_{ij}\delta^*y^i+f^2\frac{v^{\beta}}{F}g_{\alpha\beta}\delta^*v^{\alpha}
\end{equation}
The almost complex structure $J$ on $TTM$ compatible with Sasakian metric $G$ is defined by:
\begin{equation}~\label{a.c.s.}
J:=\frac{\delta^*}{\delta^*x^i}\otimes\delta^*y^i-\frac{\partial}{\partial y^i}\otimes dx^i+\frac{\delta^*}{\delta^*u^{\alpha}}\otimes\delta^*v^{\alpha}-\frac{\partial}{\partial v^{\alpha}}\otimes du^{\alpha}
\end{equation}
The \emph{horizontal Liouville} vector field of $M$ defined by $JL$ and has the local presentation as follows:
\begin{equation}~\label{H liovil}
\xi:=JL=\frac{y^i}{F}\frac{\delta^*}{\delta^*x^i}+\frac{v^{\alpha}}{F}\frac{\delta^*}{\delta^*u^{\alpha}}
\end{equation}
The horizontal Liouville vector field $\xi$ play the role of \emph{Reeb vector field}~\cite{Blair} in the present contact structure on the indicatrix bundle. The dual 1-form $\bar{\eta}$ of $\xi$ is defined by $\bar{\eta}(X)=G(\xi,X)$, and is locally expressed as below:
\begin{equation}~\label{1-form}
\bar{\eta}:=\frac{y^i}{F}g_{ij}dx^j+f^2\frac{v^{\alpha}}{F}g_{\alpha\beta}du^{\beta}
\end{equation}
The restriction of $\xi$ and $\bar{\eta}$ to the indicatrix bundle is shown by $\xi^*$ and $\eta^*$. Since, $F$ is constant and equal to the unit on the indicatrix bundle $I\!M$, equations~(\ref{H liovil}) and~(\ref{1-form}) are changed to following forms:
\begin{equation}~\label{H liovil*}
\left\{
\begin{array}{l}
\xi^*=y^i\frac{\delta^*}{\delta^*x^i}+v^{\alpha}\frac{\delta^*}{\delta^*u^{\alpha}} \cr
\eta^*=y^ig_{ij}dx^j+f^2v^{\alpha}g_{\alpha\beta}du^{\beta}
\end{array}
\right.
\end{equation}
The (1,1)-tensor field $\varphi$ is set on the indicatrix bundle $I\!M$ as follows:
\begin{equation}~\label{phi}
\varphi:=J|_{D} \ ,\ \ \ \ \varphi(\xi)=0
\end{equation}
where $D=\{X\in TM\ |\ \eta(X)=\eta^*(X)=0\}$. The distribution $D$ is called \emph{contact distribution} in contact manifolds. Also, notations $\bar{G}$ are used to restrict metric $G$ to the indicatrix bundle. Now, the following Theorem can be stated:
\begin{Theorem}~\label{contact}
Let the 4-tuple $(\varphi,\eta^*,\xi^*,\bar{G})$ be defined as above. Then the indicatrix bundle of a warped product Finsler manifold with $(\varphi,\eta^*,\xi^*,\bar{G})$ is a contact manifold.
\end{Theorem}
\begin{proof}
The compatibility of $\varphi$ and the metric $\bar{G}$ is equivalent to compatibility of $J$ and $G$. Also, The conditions
$$\eta^*\circ\varphi=0\ ,\ \ \ \ \varphi(\xi^*)=0\ ,\ \ \ \ \varphi^2=-I+\xi^*\otimes\eta^*$$
are easy to be proved by considering equations~(\ref{H liovil*}) and~(\ref{phi}). To complete the proof, the condition $d\eta^*(X,Y)=\bar{G}(X,\varphi Y)$ for the vector fields $X,Y\in\Gamma(TIM)$ needs to be checked. By calculating $d\eta^*$ the following can be obtained:\\

\centerline{$d\eta^*=-G_k^jg_{ij}dx^i\wedge dx^k+\frac{\partial g_{ij}}{\partial x^k}y^jdx^i\wedge dx^k+g_{ij}dx^i\wedge\delta^*y^j$}
\centerline{$+\frac{\partial f^2}{\partial x^k}v^{\beta}g_{\alpha\beta}du^{\alpha}\wedge dx^k-f^2B_k^{\beta}g_{\alpha\beta}du^{\alpha}\wedge dx^k-B_{\alpha}^jg_{ij}dx^i\wedge du^{\alpha}$}
\centerline{$-f^2G_{\gamma}^{\beta}g_{\beta\alpha}du^{\alpha}\wedge du^{\gamma}+f^2v^{\beta}\frac{\partial g_{\alpha\beta}}{\partial u^{\gamma}}du^{\alpha}\wedge du^{\gamma}+f^2g_{\alpha\beta}du^{\alpha}\wedge\delta^*v^{\beta}$}
\centerline{$=g_{ij}dx^i\wedge\delta^*y^j+f^2g_{\alpha\beta}du^{\alpha}\wedge\delta^*v^{\beta}$}
$$\Longrightarrow \ \ d\eta^*(X,Y)=G(X,JY)$$
Since $\bar{G}$ is the restriction of $G$ to the indicatrix and $d\eta^*(\xi,.)=0$ thus
$$d\eta^*(X,Y)=\bar{G}(X,\varphi Y)\ \ \ \ \ \forall X,Y\in\Gamma D$$
So, $I\!M$ with $(\varphi,\eta^*,\xi^*,\bar{G})$ is a contact manifold.
\end{proof}
After proving that indicatrix naturally has contact structure, it will cross the minds that in which case $I\!M$ will have Sasakian structure. In section 5, we show that indicatrix bundle cannot have Sasakian structure. Therefore, we should not have any expectation of properties of Sasakian manifolds in Riemannian geometry on indicatrix bundle as a Riemannian submanifold of $TM$.

\section{A Kahler Structure on Finslerian Warped Product Manifolds}
The first part of this section includes the Kahler structure on tangent bundle of warped product Finsler manifolds with metric $G$ introduced in~(\ref{metric}). In Theorem~\ref{contact}, it was proved that the Kahler form $\Omega$ defined by $\Omega(X,Y)=G(X,JY)$ is close. Therefore, $TM$ naturally has an almost Kahler structure. Consequently, the Kahler structure on $(M,J,G)$ is equivalent to check the integrability condition of $J$. The integrability of $J$ is equal to the vanishing of tensor field $N_J$ given by:
\begin{equation}~\label{neij.}
N_J(X,Y)=[JX,JY]-J[JX,Y]-J[X,JY]-[X,Y]\ \ \ \ \ \forall X,Y\in\Gamma TM
\end{equation}
In order to simplify the equations, the basis~(\ref{M basis1}) can be rewritten as follows:
\begin{equation}~\label{basis4}
TTM=<\frac{\delta^*}{\delta^*\textbf{x}^a},\frac{\partial}{\partial\textbf{y}^a}>
\end{equation}
where $\frac{\delta^*}{\delta^*\textbf{x}^a}=\frac{\delta^*}{\delta^* x^i}\delta_a^i+\frac{\delta^*}{\delta^*u^{\alpha}}\delta_a^{\alpha+n}$ and $\frac{\partial}{\partial\textbf{y}^a}=\frac{\partial}{\partial y^i}\delta_a^i+\frac{\partial}{\partial v^{\alpha}}\delta_a^{\alpha+n}$. Moreover, the lie bracket of horizontal basis is given by:
$$[\frac{\delta^*}{\delta^*\textbf{x}^a},\frac{\delta^*}{\delta^*\textbf{x}^b}]=R_{\ ab}^c\frac{\partial}{\partial\textbf{y}^c}$$
In which the indices are in the same range as in Section 2. For different combinations of the basis~(\ref{basis4}) in computing $N_J$ the following equations are presented:
\begin{equation}~\label{int}
\left\{
\begin{array}{l}
N_J(\frac{\delta^*}{\delta^*\textbf{x}^a},\frac{\delta^*}{\delta^*\textbf{x}^b})=-
N_J(\frac{\partial}{\partial\textbf{y}^a},\frac{\partial}{\partial\textbf{y}^b})=-R_{\ ab}^c\frac{\partial}{\partial\textbf{y}^c},\cr
N_J(\frac{\delta^*}{\delta^*\textbf{x}^a},\frac{\partial}{\partial\textbf{y}^b})=-R_{\ ab}^c\frac{\delta^*}{\delta^*\textbf{x}^c}
\end{array}
\right.
\end{equation}
From these equations, the following Corollary can be inferred:
\begin{cor}~\label{sas. kah.}
Let $(M,J,G)$ be a warped Finsler manifold. Then, $TM$ is a Kahler manifold if and only if the horizontal distribution defined by $HTM:=<\frac{\delta^*}{\delta^*\textbf{x}^1},...,\frac{\delta^*}{\delta^*\textbf{x}^{n+m}}>$ is integrable.
\end{cor}

Now, the Kahler structure on the tangent bundle of a warped Finsler manifold is studied using the Oproiu's metric~\cite{opiro1} which is used frequently in Physics phenomena. For the potential function $\tau$ defined on $F^2$ and the constants $A$ and $B$, the new metric $\tilde{G}$ is introduced as follows:
$$\tilde{G}=P_{ab}d\textbf{x}^a\otimes d\textbf{x}^b+Q_{ab}\delta^*\textbf{y}^a\otimes\delta^*\textbf{y}^b$$
where $P_{ab}=\frac{1}{A}g_{ab}+\frac{\tau(F^2)}{AB}\textbf{y}_a\textbf{y}_b$,
$Q_{ab}=Ag_{ab}-\frac{A\tau(F^2)}{B+F^2\tau(F^2)}\textbf{y}_a\textbf{y}_b$ and $\textbf{y}_a=g_{ab}\textbf{y}^b$.

It can be seen that $Q^{ab}:=g^{ac}Q_{cd}g^{db}=Ag^{ab}-\frac{A\tau(F^2)}{B+F^2\tau(F^2)}\textbf{y}^a\textbf{y}^b$ is the inverse matrix of $P_{ab}$. Also, the following $M$-tensor fields on $TM$ are obtained by the usual algebraic tensor operations,
$$\left\{\begin{array}{l}
P_a^b=P_{ac}g^{cb}=P^{bc}g_{ca}=\frac{1}{A}\delta_a^b+\frac{\tau(F^2)}{AB}\textbf{y}_a\textbf{y}^b \cr
Q_a^b=Q_{ac}g^{cb}=Q^{bc}g_{ca}=A\delta_a^b-\frac{A\tau(F^2)}{B+F^2\tau(F^2)}\textbf{y}_a\textbf{y}^b
\end{array}
\right.$$
The almost complex structure $\tilde{J}$ on $TM$ compatible with metric $\tilde{G}$ is presented as follow:
$$\tilde{J}(\frac{\partial}{\partial\textbf{y}^a}):=Q_a^b\frac{\delta^*}{\delta^*\textbf{x}^b},\ \ \ \tilde{J}(\frac{\delta^*}{\delta^*\textbf{x}^a}):=-P_a^b\frac{\partial}{\partial\textbf{y}^b}$$
\begin{Theorem}~\label{a.K.s.}
$(TM,\tilde{J},\tilde{G})$ is an almost Kahler manifold.
\end{Theorem}
\begin{proof}
First, compatibility of $\tilde{J}$ and $\tilde{G}$ are checked
$$\left\{\begin{array}{l}
\tilde{G}(\tilde{J}\frac{\partial}{\partial\textbf{y}^a},\tilde{J}\frac{\partial}{\partial\textbf{y}^b})
=Q_a^cQ_b^dP_{cd}=Q_{ae}g^{ec}g_{bf}Q^{fd}P_{cd}=Q_{ab}=\tilde{G}(\frac{\partial}{\partial\textbf{y}^a},
\frac{\partial}{\partial\textbf{y}^b})\cr
\tilde{G}(\tilde{J}\frac{\delta^*}{\delta^*\textbf{x}^a},\tilde{J}\frac{\delta^*}{\delta^*\textbf{x}^b})
=P_a^cP_b^dQ_{cd}=P_{ae}g^{ec}g_{bf}P^{fd}Q_{cd}=P_{ab}=\tilde{G}(\frac{\delta^*}{\delta^*\textbf{x}^a},
\frac{\delta^*}{\delta^*\textbf{x}^b})\end{array}\right.$$
For $(TM,\tilde{J},\tilde{G})$, the Kahler 2-form $\tilde{\Omega}$ is defined by $\tilde{\Omega}(X,Y)=\tilde{G}(X,\tilde{J}Y)$ and has local presentation as follows:
$$\left\{\begin{array}{l}
\tilde{\Omega}(\frac{\partial}{\partial\textbf{y}^a},\frac{\partial}{\partial\textbf{y}^b})=\tilde{\Omega}(
\frac{\delta^*}{\delta^*\textbf{x}^a},\frac{\delta^*}{\delta^*\textbf{x}^b})=0\cr
\tilde{\Omega}(\frac{\partial}{\partial\textbf{y}^a},\frac{\delta^*}{\delta^*\textbf{x}^b})=\tilde{G}(
\frac{\partial}{\partial\textbf{y}^a},-P_b^c\frac{\partial}{\partial\textbf{y}^c})=-P_b^cQ_{ac}=-g_{bd}
P^{dc}Q_{ac}=-g_{ab}\cr
\tilde{\Omega}(\frac{\delta^*}{\delta^*\textbf{x}^a},\frac{\partial}{\partial\textbf{y}^b})=\tilde{G}(
\frac{\delta^*}{\delta^*\textbf{x}^a},Q_b^c\frac{\delta^*}{\delta^*\textbf{x}^c})=Q_b^cP_{ac}=g_{bd}
Q^{dc}P_{ac}=g_{ab}\end{array}\right.$$
From the last equations, the following can be obtained:
$$\Longrightarrow\ \ \tilde{\Omega}=g_{ab}d\textbf{x}^a\wedge\delta^*\textbf{y}^b$$
This equation and Theorem~\ref{contact} show that $\tilde{\Omega}$ is closed and $(TM,\tilde{J},\tilde{G})$ is an almost Kahler manifold.
\end{proof}
To check that $\tilde{J}$ is a complex structure on $TM$, the local components of $N_{\tilde{J}}$ for the local basis~(\ref{M basis1}) of a warped product Finsler manifold is calculated.
\begin{equation}~\label{Nj}
\left\{
\begin{array}{l}
N_{\tilde{J}}(\frac{\delta^*}{\delta^*\textbf{x}^a},\frac{\delta^*}{\delta^*\textbf{x}^b})=(P_a^c\frac{\partial P_b^d}{\partial\textbf{y}^c}-P_b^c\frac{\partial P_a^d}{\partial\textbf{y}^c}-R_{\ ab}^d)\frac{\partial}{\partial\textbf{y}^d}\cr
+(P_{b|a}^d-P_{a|b}^d+P_b^cB_{ac}^d-P_a^cB_{cb}^d)Q_d^e\frac{\delta^*}{\delta^*\textbf{x}^e}\cr
N_{\tilde{J}}(\frac{\partial}{\partial\textbf{y}^a},\frac{\partial}{\partial\textbf{y}^b})=(Q_a^cQ_b^dR_{\ cd}^e-P_c^e\frac{\partial Q_a^c}{\partial\textbf{y}^b}+P_d^e\frac{\partial Q_b^d}{\partial\textbf{y}^a})\frac{\partial}{\partial\textbf{y}^e}\cr
+(Q_a^cQ_{b|c}^d-Q_b^cQ_{b|c}^d+Q_b^eB_{ae}^cQ_c^d-Q_a^cB_{cb}^eQ_e^d)\frac{\delta^*}{\delta^*\textbf{x}^d}\cr
N_{\tilde{J}}(\frac{\partial}{\partial\textbf{y}^a},\frac{\delta^*}{\delta^*\textbf{x}^b})=(P_b^d\frac{\partial Q_a^c}{\partial\textbf{y}^d}+\frac{\partial P_b^d}{\partial\textbf{y}^a}Q_d^c-Q_a^dQ_e^cR_{\ db}^e)\frac{\delta^*}{\delta^*\textbf{x}^c}\cr
-(Q_a^cP_{b|c}^d+Q_{a|b}^cP_c^d+Q_a^cB_{ce}^dP_b^e-B_{ab}^d)\frac{\partial}{\partial\textbf{y}^d}
\end{array}
\right.
\end{equation}
where $P_{b|c}^d=\frac{\delta^*P_b^d}{\delta^*\textbf{x}^c}$ and $Q_{a|b}^c=\frac{\delta^*Q_a^c}{\delta^*\textbf{x}^b}$. It can be obtained that $N_{\tilde{J}}(\frac{\delta^*}{\delta^*\textbf{x}^a},\frac{\delta^*}{\delta^*\textbf{x}^b})=0$ implies $N_{\tilde{J}}(\frac{\partial}{\partial\textbf{y}^a},\frac{\partial}{\partial\textbf{y}^b})=N_{\tilde{J}}(
\frac{\partial}{\partial\textbf{y}^a},\frac{\delta^*}{\delta^*\textbf{x}^b})=0$. Therefore, the following corollary can be expressed.
\begin{cor}~\label{c.s.}
Let $(TM,\tilde{J},\tilde{G})$ be the tangent manifold of warped Finsler manifold $(M,F)$. Then, $(TM,\tilde{J},\tilde{G})$ is Kahler manifold if and only if the followings hold:\\
$$P_a^c\frac{\partial P_b^d}{\partial\textbf{y}^c}-P_b^c\frac{\partial P_a^d}{\partial\textbf{y}^c}-R_{\ ab}^d=
P_{b|a}^d-P_{a|b}^d+P_b^cB_{ac}^d-P_a^cB_{cb}^d=0$$
\end{cor}
\begin{proof}
It is an obvious conclusion of equations~(\ref{Nj}) and Theorem~\ref{a.K.s.}.
\end{proof}

\section{Sasakian Structure and Indicatrix Bundle}
In this section, the components of Levi-Civita connection on the Indicatrix bundle of Finsler manifolds are computed. In the first part, a new frame is set on a Finsler manifold which decomposes $TTM$ in the following way:

\centerline{$TTM=<L>\oplus<\xi>\oplus<D>$}
\noindent where $L$ and $\xi$ are presented by $y^i\frac{\partial}{\partial y^i}$ and $y^i\frac{\delta}{\delta x^i}$, respectively. The next part includes the Sasakian structure on the indicatrix bundle of a Finsler manifold.

\subsection{Levi-Civita Connection on Indicatrix Bundle}
Supposed $(M,F)$ as an $n$-dimensional Finsler manifold. Since the vertical distribution $VM$ is integrable in $TTM$ and vertical Liouville vector field $L$ is a foliation in $TTM$ which belongs to $VM$, therefore, orthogonal distribution $L'M$ to $L$ in $VM$ is a foliation and the local basis can be set as follows:
$$L'M=<\frac{\bar{\partial}}{\bar{\partial}y^1},...,\frac{\bar{\partial}}{\bar{\partial}y^{n-1}}>$$
where $\frac{\bar{\partial}}{\bar{\partial}y^a}=E_a^i\frac{\partial}{\partial y^i}\ \ \ \forall a=1,...,n-1$
and $E_a^i$ is the $(n-1)\times n$ matrix of maximum rank. The first property of this matrix is $E_a^ig_{ij}y^j=0$
achieved by the feature $G(\frac{\bar{\partial}}{\bar{\partial}y^a},L)=0$. Now, using the natural almost complex structure $J$ on $TTM$, the new basis for $TTM$ is introduced as follows:
\begin{equation}~\label{new bas}
TTM=T(I\!M)\oplus <L>=<\frac{\bar{\delta}}{\bar{\delta}x^a},\ \xi,\ \frac{\bar{\partial}}{\bar{\partial}y^a},\ L>
\end{equation}
where $\frac{\bar{\delta}}{\bar{\delta}x^a}:=J\frac{\bar{\partial}}{\bar{\partial}y^a}$. The Sasakian metric $G$ on $TM$ can be shown in the new coordinate system~(\ref{new bas}) as follows:
\begin{equation}~\label{met.mat.}
G:=\left(%
\begin{array}{cccc}
  g_{ab} & 0 & 0 & 0\\
  0 & F^2 & 0 & 0\\
  0 & 0 & g_{ab} & 0\\
  0 & 0 & 0 & F^2\\
\end{array}%
\right)
\end{equation}
where $a,b\in\{1,...,n-1\}$ and $g_{ab}=G(\frac{\bar{\delta}}{\bar{\delta}x^a},\frac{\bar{\delta}}{\bar{\delta}x^b})
=G(\frac{\bar{\partial}}{\bar{\partial} y^a},\frac{\bar{\partial}}{\bar{\partial}y^b})=g_{ij}E_a^iE_b^j$. Now, Lie brackets of the basis~(\ref{new bas}) are presented as follows:
\begin{equation}~\label{lie bracket}
\left\{
\begin{array}{l}
(1) \ [\frac{\bar{\delta}}{\bar{\delta}x^a},\frac{\bar{\delta}}{\bar{\delta}x^b}]=(\frac{\bar{\delta}E_b^i}{\bar{\delta}x^a}
-\frac{\bar{\delta}E_a^i}{\bar{\delta}x^b})\frac{\delta}{\delta x^i}+E_a^iE_b^jR_{ij}^k\frac{\partial}{\partial y^k},\\
(2) \ [\frac{\bar{\delta}}{\bar{\delta}x^a},\frac{\bar{\partial}}{\bar{\partial}y^b}]=(\frac{\bar{\delta}E_b^k}{\bar{\delta}x^a}
+E_a^iE_b^jG_{ij}^k)\frac{\partial}{\partial y^k}-\frac{\bar{\partial}E_a^i}{\bar{\partial}y^b}\frac{\delta}{\delta x^i},\\
(3) \
[\frac{\bar{\partial}}{\bar{\partial}y^a},\frac{\bar{\partial}}{\bar{\partial}y^b}]=(
\frac{\bar{\partial}E_b^i}{\bar{\partial}y^a}-\frac{\bar{\partial}E_a^i}{\bar{\partial}y^b})\frac{\partial}{\partial y^i},\\
(4) \
[\frac{\bar{\delta}}{\bar{\delta}x^a},\xi]=-(E_a^iG_i^j+\xi(E_a^j))\frac{\delta}{\delta x^j}+E_a^iy^jR_{ij}^k
\frac{\partial}{\partial y^k},\\
(5) \
[\frac{\bar{\partial}}{\bar{\partial}y^a},\xi]=\frac{\bar{\delta}}{\bar{\delta}x^a}-(\xi(E_a^j)+E_a^iG_i^j)
\frac{\partial}{\partial y^j},\\
(6) \
[\frac{\bar{\delta}}{\bar{\delta}x^a},L]=-L(E_a^i)\frac{\delta}{\delta x^i},\\
(7) \
[\frac{\bar{\partial}}{\bar{\partial}y^a},L]=\frac{\bar{\partial}}{\bar{\partial}y^a}-L(E_a^i)
\frac{\partial}{\partial y^i},\\
(8) \
[\xi,\xi]=[L,L]=[\xi,L]+\xi=0.
\end{array}
\right.
\end{equation}
The local components of Levi-Civita connection $\nabla$ given by:
\begin{equation}~\label{levi-civita}
\left\{\begin{array}{l}
2G(\nabla_XY,Z)=XG(Y,Z)+YG(X,Z)-ZG(X,Y)\cr
-G([X,Z],Y)-G([Y,Z],X)+G([X,Y],Z)
\end{array}
\right.
\end{equation}
for basis~(\ref{new bas}) and metric~(\ref{met.mat.}) are expressed as follows:
\begin{equation}~\label{levicivita1}
\left\{
\begin{array}{l}
\nabla_{\frac{\bar{\delta}}{\bar{\delta}x^a}}\frac{\bar{\delta}}{\bar{\delta}x^b}=(\Gamma_{ab}^e+
\frac{\bar{\delta}E_b^j}{\bar{\delta}x^a}E_d^kg_{jk}g^{de})
{\frac{\bar{\delta}}{\bar{\delta}x^e}}+(-g_{ab}^e+\frac{1}{2}R_{ab}^e)\frac{\bar{\partial}}{\bar{\partial}y^e}
+\frac{1}{2F^2}\bar{R}_{ab}\xi,\\
\nabla_{\frac{\bar{\delta}}{\bar{\delta}x^a}}\frac{\bar{\partial}}{\bar{\partial}y^b}
=\left(\frac{1}{2}E_b^jE_d^kE_a^i(\frac{\delta g_{jk}}{\delta x^i}-G_{ik}^hg_{hj}+G_{ij}^hg_{hk})+\frac{\bar{\delta}E_b^j}{\bar{\delta}x^a}
E_d^kg_{jk}\right)g^{de}\frac{\bar{\partial}}{\bar{\partial}y^e}\cr +(g_{ab}^e-\frac{1}{2}R_{bad}g^{de})\frac{\bar{\delta}}{\bar{\delta}x^e}
+\frac{1}{2F^2}R_{ab}\xi,\\
\nabla_{\frac{\bar{\partial}}{\bar{\partial}y^b}}\frac{\bar{\delta}}{\bar{\delta}x^a}=(g_{ab}^e-
\frac{1}{2}R_{bad}g^{de}+\frac{\bar{\partial}E_a^i}{\bar{\partial}y^b}E_d^kg_{ik}g^{de}
)\frac{\bar{\delta}}{\bar{\delta}x^e}+\frac{1}{F^2}(\frac{1}{2}R_{ab}-g_{ab})\xi\cr
+\frac{1}{2}E_a^iE_b^jE_d^k(\frac{\delta g_{jk}}{\delta x^i}-G_{ik}^hg_{hj}
-G_{ij}^hg_{hk})g^{de}\frac{\bar{\partial}}{\bar{\partial}y^e},\\
\nabla_{\frac{\bar{\partial}}{\bar{\partial}y^a}}\frac{\bar{\partial}}{\bar{\partial}y^b}=\frac{1}{2}E_a^iE_b^jE_d^k
(G_{ik}^hg_{hj}+G_{jk}^hg_{hi}-\frac{\delta g_{ij}}{\delta x^k})g^{de}\frac{\bar{\delta}}{\bar{\delta}x^e}\cr +(g_{ab}^e+\frac{\bar{\partial}E_b^j}{\bar{\partial}y^a}E_d^kg_{kj}g^{de})\frac{\bar{\partial}}{\bar{\partial}y^e}
-\frac{1}{F^2}g_{ab}L,\\
\nabla_{\frac{\bar{\delta}}{\bar{\delta}x^a}}\xi=\frac{1}{2}\bar{R}_{da}g^{de}
\frac{\bar{\delta}}{\bar{\delta}x^e}-\frac{1}{2}R_{ad}g^{de}\frac{\bar{\partial}}{\bar{\partial}y^e},\\
\nabla_{\xi}\frac{\bar{\delta}}{\bar{\delta}x^a}=(\xi(E_a^i)E_d^kg_{ik}+E_a^iG_i^hg_{hk}E_d^k+\frac{1}{2}\bar{R}_{da})
g^{de}\frac{\bar{\delta}}{\bar{\delta}x^e}+\frac{1}{2}R_{ad}g^{de}\frac{\bar{\partial}}{\bar{\partial}y^e},\\
\nabla_{\frac{\bar{\partial}}{\bar{\partial}y^a}}\xi=(\delta_a^e-\frac{1}{2}R_{ad}g^{de})
\frac{\bar{\delta}}{\bar{\delta}x^e},\\
\nabla_{\xi}\frac{\bar{\partial}}{\bar{\partial}y^a}=-\frac{1}{2}R_{ad}g^{de}\frac{\bar{\delta}}{\bar{\delta}x^e}
+(\xi(E_a^i)g_{ik}+E_a^iG_i^jg_{jk})E_d^kg^{de}\frac{\bar{\partial}}{\bar{\partial}y^e},\\
\nabla_{\frac{\bar{\delta}}{\bar{\delta}x^a}}L=\nabla_L\frac{\bar{\delta}}{\bar{\delta}x^a}-L(E_a^i)E_d^kg_{ik}g^{de}
\frac{\bar{\delta}}{\bar{\delta}x^e}=0,\\
\nabla_{\frac{\bar{\partial}}{\bar{\partial}y^a}}L
-\frac{\bar{\partial}}{\bar{\partial}y^a}=\nabla_L\frac{\bar{\partial}}{\bar{\partial}y^a}-L(E_a^i)E_d^kg_{ik}g^{de}
\frac{\bar{\partial}}{\bar{\partial}y^e}=0,\\
\nabla_{\xi}\xi=\nabla_{\xi}L=\nabla_L\xi-\xi=\nabla_LL-L=0.
\end{array}
\right.
\end{equation}
where

\centerline{$g_{ab}^c=g_{abd}g^{dc}=\frac{1}{2}E_a^iE_b^jE_d^kg_{ijk}g^{dc}\ ,\ \ \ \ \ \ \Gamma_{ab}^c=E_a^iE_b^jE_d^k\Gamma_{ij}^hg_{hk}g^{dc}$}
\centerline{$R_{ab}^c=R_{dab}g^{dc}=E_a^iE_b^jE_d^kR_{ij}^hg_{hk}g^{dc},$}
\centerline{$\bar{R}_{ab}=(\frac{\bar{\delta}E_b^i}{\bar{\delta}x^a}
-\frac{\bar{\delta}E_a^i}{\bar{\delta}x^b})g_{ij}y^j\ ,\ \ \ \ \ \ R_{ab}=E_a^iE_b^jR_{ij}$}
\noindent and, $(g^{ab})$ is the inverse matrix of $(g_{ab})$.

In following, the Levi-Civita connection and metric on indicatrix bundle are denoted by $\bar{\nabla}$ and $\bar{G}$, respectively, which $\bar{G}$ is the restriction of metric~(\ref{met.mat.}). In order to compute the components of Levi-Civita connection $\bar{\nabla}$ on indicatrix bundle $I\!M$ for the basis~(\ref{new bas}) the \emph{Guass Formula}~\cite{lee}:
\begin{equation}~\label{guass}
\nabla_XY=\bar{\nabla}_XY+H(X,Y)
\end{equation}
where $H$ is the \emph{second fundamental form} of $I\!M$ is needed. It is obvious that all components in~(\ref{levicivita1}) and except $\nabla_{\frac{\bar{\partial}}{\bar{\partial}y^a}}\frac{\bar{\partial}}{\bar{\partial}y^b}$ are tangent to $I\!M$. Therefore, $\bar{\nabla}$ is equal to $\nabla$ for the other components of~(\ref{levicivita1}) by using the Gauss formula~(\ref{guass}). The following corollary can be stated for the component $\nabla_{\frac{\bar{\partial}}{\bar{\partial}y^a}}\frac{\bar{\partial}}{\bar{\partial}y^b}$.
\begin{cor}~\label{tot. geo.}
The indicatrix bundle $I\!M$ is not a totally geodesic submanifold of $TM$.
\end{cor}
\begin{proof}
From~(\ref{levicivita1}), we obtain
$H(\frac{\bar{\partial}}{\bar{\partial}y^a},\frac{\bar{\partial}}{\bar{\partial}y^b})=-\frac{1}{F^2}g_{ab}L$
which it cannot be vanish. Therefore, $I\!M$ is not totally geodesic submanifold of $TM$.
\end{proof}
The curvature tensor $R$ of $\nabla$ defined by $R(X,Y)Z=\nabla_X\nabla_YZ-\nabla_Y\nabla_XZ-\nabla_{[X,Y]}Z$
is related to the curvature tensor $\bar{R}$ of $\bar{\nabla}$ in following equations:
\begin{equation}~\label{cur}
\left\{
\begin{array}{l}
R(\frac{\bar{\delta}}{\bar{\delta}x^a},\frac{\bar{\delta}}{\bar{\delta}x^b})\frac{\bar{\partial}}{\bar{\partial}y^c}
=\bar{R}(\frac{\bar{\delta}}{\bar{\delta}x^a},\frac{\bar{\delta}}{\bar{\delta}x^b})
\frac{\bar{\partial}}{\bar{\partial}y^c}+\frac{1}{F^2}R_{cab}L\\
R(\frac{\bar{\delta}}{\bar{\delta}x^a},\frac{\bar{\partial}}{\bar{\partial}y^b})\frac{\bar{\delta}}{\bar{\delta}x^c}
=\bar{R}(\frac{\bar{\delta}}{\bar{\delta}x^a},\frac{\bar{\partial}}{\bar{\partial}y^b})
\frac{\bar{\delta}}{\bar{\delta}x^c}+\frac{1}{2F^2}(R_{bac}-2g_{abc})L\\
R(\frac{\bar{\partial}}{\bar{\partial}y^a},\frac{\bar{\partial}}{\bar{\partial}y^b})
\frac{\bar{\partial}}{\bar{\partial}y^c}
=\bar{R}(\frac{\bar{\partial}}{\bar{\partial}y^a},\frac{\bar{\partial}}{\bar{\partial}y^b})
\frac{\bar{\partial}}{\bar{\partial}y^c}-\frac{1}{F^2}g_{bc}\frac{\bar{\partial}}{\bar{\partial}y^a}+
\frac{1}{F^2}g_{ac}\frac{\bar{\partial}}{\bar{\partial}y^b}\\
R(\frac{\bar{\delta}}{\bar{\delta}x^a},\frac{\bar{\partial}}{\bar{\partial}y^b})
\frac{\bar{\partial}}{\bar{\partial}y^c}
=\bar{R}(\frac{\bar{\delta}}{\bar{\delta}x^a},\frac{\bar{\partial}}{\bar{\partial}y^b})
\frac{\bar{\partial}}{\bar{\partial}y^c}+\frac{1}{2F^2}E_a^iE_b^jE_d^k
(G_{ik}^hg_{hj}+G_{jk}^hg_{hi}-\frac{\delta g_{ij}}{\delta x^k})L\\
R(\frac{\bar{\delta}}{\bar{\delta}x^a},\frac{\bar{\partial}}{\bar{\partial}y^b})\xi=
\bar{R}(\frac{\bar{\delta}}{\bar{\delta}x^a},\frac{\bar{\partial}}{\bar{\partial}y^b})\xi-\frac{1}{2F^2}R_{ab}L\\
R(\frac{\bar{\partial}}{\bar{\partial}y^a},\xi)\frac{\bar{\delta}}{\bar{\delta}x^b}=
\bar{R}(\frac{\bar{\partial}}{\bar{\partial}y^a},\xi)\frac{\bar{\delta}}{\bar{\delta}x^b}-\frac{1}{2F^2}R_{ab}L\\
R(\frac{\bar{\delta}}{\bar{\delta}x^a},\xi)\frac{\bar{\partial}}{\bar{\partial}y^b}=
\bar{R}(\frac{\bar{\delta}}{\bar{\delta}x^a},\xi)\frac{\bar{\partial}}{\bar{\partial}y^b}-\frac{1}{F^2}R_{ab}L
\end{array}
\right.
\end{equation}
where $G_{abc}=E_a^iE_b^jE_c^kG_{ij}^hg_{hk}$. For the other combinations of $\frac{\bar{\delta}}{\bar{\delta}x^a},\frac{\bar{\partial}}{\bar{\partial}y^a}$ and $\xi$, tensor fields $R$ and $\bar{R}$ coincide with each other.

\subsection{Sasakian Structure and Indicatrix Bundle of a Finsler Manifold}
First, let $(M,\varphi,\eta,\xi,g)$ be a contact Riemannian manifold. In~\cite{bejancu2}, it was proved that $M$ is Sasakain manifold if and only if
\begin{equation}~\label{sasa1}
(\tilde{\nabla}_X\varphi)Y=0 \ \ \ \ \forall X,Y\in\Gamma(TM)
\end{equation}
where
$$\tilde{\nabla}_XY=\nabla_XY-\eta(X)\nabla_Y\xi-\eta(Y)\nabla_X\xi+(d\eta+\frac{1}{2}(\mathcal{L}_{\xi}g))(X,Y)\xi$$
and $\nabla$ is Levi-Civita connection on $(M,g)$. Since the indicatrix bundle has the contact metric structure in Finslerian manifolds by Proposition 4.1 in~\cite{bejancu2008} and in warped product Finsler manifolds by Theorem~\ref{contact}, here it is tried to find an answer to the question that "\emph{Can the indicatrix bundle be a Sasakian manifold with restricted sasaki metric $G$ and almost complex structure $J$ on $TM$?}". First, the following Lemma is proved in order to reduce the number of calculations.
\begin{lemma}\label{redu}
If $(M,\varphi,\eta,\xi,g)$ be a contact metric manifold with contact distribution $D$, then $M$ is Sasakian manifold if and only if:
$$(\tilde{\nabla}_X\varphi)Y=0 \ \ \ \ \forall X,Y\in\Gamma(D)$$
\end{lemma}
\begin{proof}
For all $\bar{X}\in\Gamma(TM)$, they can be written in the form $X+f\xi$ where $X\in\Gamma D$ and $f\in C^{\infty}(M)$. Therefore:

\centerline{$(\tilde{\nabla}_{\bar{X}}\varphi)\bar{Y}=(\tilde{\nabla}_{X+f\xi}\varphi)(Y+g\xi)=
(\tilde{\nabla}_X\varphi)Y+(\tilde{\nabla}_{f\xi}\varphi)Y+(\tilde{\nabla}_X\varphi)g\xi$}
\centerline{$+(\tilde{\nabla}_{f\xi}\varphi)g\xi=(\tilde{\nabla}_X\varphi)Y+f(\tilde{\nabla}_{\xi}\varphi Y-\varphi\tilde{\nabla}_{\xi}Y)+\tilde{\nabla}_X\varphi(g\xi)-\varphi(\tilde{\nabla}_Xg\xi)$}
\centerline{$+f(\tilde{\nabla}_{\xi}\varphi{g\xi}-\varphi\tilde{\nabla}_{\xi}g\xi)=(\tilde{\nabla}_X\varphi)Y$}
\noindent The lemma is proved using Theorem 3.2 in~\cite{bejancu2} and the last equation.
\end{proof}
Now, the following Theorem can be expressed:
\begin{Theorem}~\label{sasaki}
Let $(M,F)$ be a Finsler manifold. Then, indicatrix bundle $I\!M(1)$ with its contact structure given in Theorem~\ref{contact} and Proposition 4.1 in~\cite{bejancu2008} can never be a Sasakian manifold.
\end{Theorem}
\begin{proof}
From lemma~\ref{redu}, $I\!M$ is a Sasakian manifold if and only if:

\centerline{$(\tilde{\nabla}_{\frac{\bar{\delta}}{\bar{\delta}x^a}}\varphi)\frac{\bar{\delta}}{\bar{\delta}x^b}= (\tilde{\nabla}_{\frac{\bar{\delta}}{\bar{\delta}x^a}}\varphi)\frac{\bar{\partial}}{\bar{\partial}y^b}= (\tilde{\nabla}_{\frac{\bar{\partial}}{\bar{\partial}y^a}}\varphi)\frac{\bar{\delta}}{\bar{\delta}x^b}= (\tilde{\nabla}_{\frac{\bar{\partial}}{\bar{\partial}y^e}}\varphi)\frac{\bar{\partial}}{\bar{\partial}y^b}=0$}
\noindent Using~(\ref{levicivita1}), one of the components in above equations is

\centerline{$g_{ab}=0$}
\noindent which demonstrates a contradiction and shows that the indicatrixe bundle cannot be in the Sasakian structure.
\end{proof}


\noindent\textbf{Another proof for Theorem~\ref{sasaki}}\\
The following argument was presented in~\cite{galva} and Chapter 6 in~\cite{Blair}. We consider $TM=I\!M\times\mathbb{R}$ for the Finslerian warped product manifold $(M,F)$ and introduce the almost complex structure $\bar{J}$ by means of $\varphi$ defined in~(\ref{phi}) as follows:
$$\bar{J}(X+fL)=\varphi(X)-f\xi+\eta(X)L$$
where $X$ is a vector field tangent to indicatrix bundle and $\xi,\eta$ were defined in Section 3. Using a straight calculation, it can be seen that $\bar{J}$ is equal to $J$ defined in~(\ref{a.c.s.}). The contact structure $(\varphi,\eta,\xi)$ will be a Sasakian structure if and only if $(\varphi,\eta,\xi)$ is normal, or equivalently, $\bar{J}$ is integrable. From~(\ref{int}), it can be inferred that $\bar{J}$ is integrable if and only if $M$ is a flat manifold. Up to now, it can be shown that $I\!M$ is Sasakian if and only if $M$ is flat. Furthermore, it was proved that $\xi$ is a killing vector field on each Sasakian manifold~\cite{Blair}. Therefore, $M$ has constant curvature 1 using  Theorem 3.4 in~\cite{bejancu} and it is a contradiction to the previous result which shows that $M$ is flat if $I\!M$ is a Sasakian manifold. $\Box$


Hassan Attarchi\\
Department of Mathematics and Computer Science,\\
Amirkabir University of Technology, Tehran, Iran.
E-mail: hassan.attarchi@aut.ac.ir\\

\noindent Corresponding author: Morteza Mir Mohammad Rezaii;\\
Department of Mathematics and Computer Science,\\
Amirkabir University of Technology, Tehran, Iran.
E-mail: mmreza@aut.ac.ir\\

\begin{thebibliography}{00}
\bibitem{Rezaii} Y. Alipour-Fakhri and M. M. Rezaii, The warped Sasaki-Matsumoto metric and bundlelike condition,
Journal of Mathematical Physics, 51 (2010), 122701-1$\thicksim$122701-13.
\bibitem{balan} V. Balan, E. Payghan and A. Tayebi, Structures of the indicatrix bundle of Finsler-Rizza manifolds,
Balkan Journal of Geometry and Its Applications, 16 (2011), 1-12.
\bibitem{bao-chern-shen} D. Bao, S.S. Chen and Z. Shen, An Introduction to Riemann-Finsler Geometry, Springer,
New York, 2000.
\bibitem{bejancu2008} A. Bejancu, Tangent Bundle and Indicatrix Bundle of a Finsler Manifold,
Journal of Kodai Mathematics, 31 (2008), 272-306.
\bibitem{bejancu1} A. Bejancu, Oblique warped products, Journal of Geometry and Physics, 57 (2007), 1055-1073.
\bibitem{bejancu2} A. Bejancu, Kahler contact distribution,
Journal Geometry and Physics, 60 (2010), 1958-1967.
\bibitem{bejancu} A. Bejancu and H. R. Farran, Finsler Geometry and Natural Folitions on the Tangent Bundle,
Reports on Mathematical Physics, 58 (2006), 131-146.
\bibitem{chen} B. Y. Chen, Geometry of Warped Product CR-Submanifolds in Kaehler Manifolds,
Monatsh. Math., 133 (2001), 177-195.
\bibitem{onil} R. Bishop and B. O'Neill, Manifolds of negative curvature, Trans. Amer. Math. Soc., 46 (1969), 1-49.
\bibitem{Blair} D. E. Blair, Riemannian Geometry of Contact and Symplectic Manifolds,
Birkhauser, Boston, 2002.
\bibitem{galva} G. Calvaruso and D. Perrone, Contact pseudo-metric manifolds, Diff.
Geom. Applic., 28 (2010), 615-634.
\bibitem{kosma} L. Kosma, I. R. Peter and C. Varga, Warped product of Finsler-manifolds,
Ann. Univ. Sci. Pudapest, 44 (2001), 157-170.
\bibitem{lee} J.M. Lee, Riemannian Manifolds: An Introduction to Curvature,
Springer, New York, 1997.
\bibitem{matsumoto} M. Matsumoto, Foundations of Finsler geometry and special Finsler spaces,
Kaiseisha Press, Saikawa Otsu, 1986.
\bibitem{onil2} B. O'Neill, Semi-riemannian geometry with
application to relativity, Academic Press, New York, 1983.
\bibitem{opiro1} V. Oproiu and N. Papaghiuc, A Kähler structure on the nonzero tangent bundle of a space form, Diff.
Geom. Applic., 11 (1999), 1-12.
\bibitem{payghan} E. Peyghan and A. Tayebi, A Kahler structure on Finsler spaces with nonzero constant flag curvature, Journal of Mathematical Physics, 51 (2010), 022904-1$\thicksim$022904-11.

\end{thebibliography}
\end{document}